\documentclass[11pt,leqno,oneside,a4paper]{amsart}
\usepackage{amsmath,amssymb,amsthm,scalefnt}
\usepackage[alphabetic,abbrev]{amsrefs} 
\usepackage[centertableaux]{ytableau}
 \usepackage{tikz}

\makeatletter 
\def\section{\@startsection{section}{1}%
\z@{.7\linespacing\@plus\linespacing}{.5\linespacing}%
{\normalfont\bfseries\centering\large}}
\makeatother

\newcommand{\dom}{\unrhd}  
\newcommand{\notdom}{\ntrianglerighteq} 
\newcommand{\lessdom}{\unlhd}  
\newcommand{\slessdom}{\lhd}  
\newcommand{\notlessdom}{\ntrianglelefteq}  
\newcommand{\stt}{\mathsf{t}}  
\newcommand{\sts}{\mathsf{s}}  

\DeclareRobustCommand{\stirling}{\genfrac\{\}{0pt}{}} 

\newcommand{\C}{{\mathbb C}}

\newcommand{\Ptn}{\mathcal{P}}   
\newcommand{\End}{\operatorname{End}}

\newcommand{\GL}{\operatorname{GL}}

 \newcommand{\Sym}{\mathfrak{S}}
\newcommand{\Symset}{\operatorname{Sym}}

\newcommand{\sgn}{\operatorname{sgn}}

\newcommand{\ann}{\operatorname{ann}}
\newcommand{\rows}{\operatorname{rows}}
\newcommand{\cols}{\operatorname{cols}}
\newcommand{\bV}{\mathbf{V}} 
\newcommand{\bv}{\mathbf{v}} 
\newcommand{\rank}{\operatorname{rank}} 
\newcommand{\Hook}{\mathcal{H}} 
\newcommand{\eps}{\varepsilon} 
\newcommand{\op}{\operatorname{op}} 
\newcommand{\ind}{\operatorname{ind}} 
\newcommand{\im}{\operatorname{im}} 
 \newcommand{\tr}{\prime} 

 \newtheorem{thm}{Theorem}[section]  
\newtheorem{lem}[thm]{Lemma}
\newtheorem{prop}[thm]{Proposition}
\newtheorem{cor}[thm]{Corollary}
\newtheorem*{thm*}{Theorem}
\newtheorem*{lem*}{Lemma}
\newtheorem*{prop*}{Proposition}
\newtheorem*{cor*}{Corollary}

\theoremstyle{definition}
\newtheorem{defn}[thm]{Definition}
\newtheorem{example}[thm]{Example}

\newtheorem*{defn*}{Definition}
\newtheorem*{example*}{Example}
\newtheorem*{examples*}{Examples}

 \newtheorem{rmk}[thm]{Remark}

\newtheorem*{rmk*}{Remark}
\newtheorem*{rmks*}{Remarks}

\allowdisplaybreaks

\begin{document}
\title[Second fundamental theorem for partition algebras]{An integral
  second fundamental theorem of invariant theory for partition
  algebras}

  \author{Chris Bowman}
       \address{Department of Mathematics, 
University of York, Heslington, York, YO10 5DD, UK}
\email{Chris.Bowman-Scargill@york.ac.uk}

  \author{Stephen Doty}
  \address{Department of Mathematics and Statistics, Loyola University
    Chicago, Chicago, IL 60660 USA}
  \email{doty@math.luc.edu}

  \author{Stuart Martin}
  \address{DPMMS, Centre for Mathematical Sciences, Wilberforce Road,
  Cambridge, CB3 0WB, UK}
  \email{S.Martin@dpmms.cam.ac.uk}

\begin{abstract}\noindent 
We prove that the kernel of the action of the group algebra of the
Weyl group acting on tensor space (via restriction of the action from
the general linear group) is a cell ideal with respect to the
alternating Murphy basis. This provides an analogue of the second
fundamental theory of invariant theory for the partition algebra over
an arbitrary commutative ring and proves that the centraliser algebras
of the partition algebra are cellular.  We also prove similar results
for the half partition algebras.
\end{abstract}
\maketitle

\section*{Introduction}\noindent
The partition algebra $\Ptn_r(n)$ over the complex field $\C$ arose in
work of Paul Martin \cites{Marbook,Martin:94} and (independently)
Vaughan Jones \cite{Jones} as a generalisation of the Temperley--Lieb
algebra for $n$-state $r$-site Potts models in statistical mechanics.
Suppose that $\bV$ is an $n$-dimensional complex vector space.  The
algebra $\Ptn_r(n)$ arises as the (generic) centraliser 
  of the   group of permutation matrices  $W_n\leq \GL(\bV)$ 
acting  
\color{black} 
 on tensor space $\bV^{\otimes r}$. By the
general theory of finite dimensional algebras it follows that the $(\C
W_n, \Ptn_r(n))$-bimodule $\bV^{\otimes r}$ satisfies Schur--Weyl
duality (see \cite{Halverson-Ram}), in the sense that the image of
each representation coincides with the centraliser algebra of the the
other action.

The partition algebra has found surprising applications to Deligne's
tensor categories (see \cites{Deligne,CO11,CW12,CO14}) and the study of the
Kronecker coefficients (see \cite{BDO15}).  Heuristically speaking,
this is because the partition algebra controls stability phenomena
arising in the representation theory of symmetric groups. 

More generally, let $\bV$ be a free $\Bbbk$-module of rank $n$ over an
arbitrary (unital) commutative ring $\Bbbk$. The partition algebra makes
sense as an algebra over $\Bbbk$, and $\bV^{\otimes r}$ is a $(\Bbbk
W_n, \Ptn_r(n))$-bimodule. In a companion paper \cite{BDM:SWD}, the
authors have shown that Schur--Weyl duality holds over $\Bbbk$.
Therefore, it is natural to expect that the partition algebra will
continue to influence stability phenomena of symmetric groups over
fields of positive characteristic.

The main result of this paper, Theorem \ref{thm:main}, is that for any
commutative ring $\Bbbk$, the annihilator of the $\Bbbk W_n$-action on
$\bV^{\otimes r}$ is a cell ideal with respect to the alternating
Murphy basis of $\Bbbk W_n$. (The theory of cellular algebras was
introduced in \cite{Graham-Lehrer}.) In light of Schur--Weyl duality,
our main result implies (Corollary \ref{cor:main}) that the
centraliser algebra $\End_{\Ptn_r(n)}(\bV^{\otimes r})$ inherits a
cellular structure from that of $\Bbbk W_n$ --- even better, we obtain
an explicit cellular basis.  Thus our main result provides an analogue
of the second fundamental theory of invariant theory for the partition
algebra over an arbitrary commutative ring $\Bbbk$. Similar results for
other diagram algebras have been obtained in
\cites{Haerterich,MR2979865,BEG}.  

One can ask for conditions under which a centraliser algebra of a
cellular algebra is again cellular; this question seems to be poorly
understood in general. Our result establishes another positive
occurrence of such a phenomenon. In a different direction, the
cellularity of the centraliser algebra $\End_{W_n}(\bV^{\otimes r})$
was only recently established \cite{Donkin}; see also
\cites{Halverson-Ram, BH:1, BH:2} for explicit descriptions of the
annihilator of the action of $\Ptn_r(n)$.
 
Finally, we remark that Paul Martin \cite{Martin:2000} introduced the
half partition algebras $\Ptn_{r+\,1/2}(n)$ in order to collate the
individual (ordinary and half) partition algebras together in a tower
of recollement structure. Our results treat both algebras
$\End_{\Ptn_r(n)}(\bV^{\otimes r})$ and
$\End_{\Ptn_{r+\,1/2}(n)}(\bV^{\otimes r})$ uniformly.

\subsection*{Acknowledgements.} The authors thank the
organisers of the conference ``Representation theory of symmetric
groups and related algebras'' held at the Institute for Mathematical
Sciences, National University of Singapore (December 2017), for
providing an excellent venue for working on this paper, and Stephen
Donkin for useful comments on an earlier version.

\section{Combinatorics of symmetric groups}\label{sec:comb}
\noindent
We write $\Symset_S$ for the \emph{symmetric group} of permutations of
a set $S$ (the bijections of $S$ under composition). We write
$\Symset_d$ for $\Symset_{\{1, \dots, d\}}$. For any set $S$ with
$|S|=d$ we identify $\Symset_S$ with $\Symset_d$ via the obvious
isomorphism.    We let     $*$ denote the 
  anti-involution  which sends  $w$ to
  $w^{-1}$, for any $w\in \Symset_d$ and we extend this $\Bbbk$-linearly to the group algebra. 
  \color{black}

A \emph{weak composition} of a non-negative integer $d$ is a way of
writing $d$ as the sum of a sequence of non-negative integers. There
are infinitely many weak compositions of $d$, because we can always
append 0 to any weak composition. Weak compositions are usually
identified with infinite sequences with finite support (finitely many
non-zero terms). The \emph{length} of a weak composition $\lambda=(\lambda_1,
\lambda_2, \dots)$ is the largest $\ell$ for which $\lambda_\ell \ne
0$. Thus we write $\lambda$ as $\lambda = (\lambda_1, \dots,
\lambda_\ell)$. There are finitely many weak compositions with a
specified upper bound on length.

A \emph{composition} of $d$ is a way of writing $d$ as the sum of a
sequence of (strictly) positive integers. So a composition is a weak
composition with positive parts, and its length is the number of
parts. We stipulate that the integer 0 has one composition, of length
0, defined by the empty sequence. We write $\lambda \vDash d$ to mean
that $\lambda = (\lambda_1, \dots, \lambda_\ell)$ is a composition of
$d$.  

 Two sequences that differ in the order of their terms define different
compositions of their sum, while they are considered to define the
same \emph{partition} of that number. Thus, partitions may be
identified with ordered compositions $(\lambda_1, \dots,
\lambda_\ell)$ satisfying $\lambda_1 \ge \cdots \lambda_{\ell-1} \ge
\lambda_\ell$. We write $\lambda \vdash d$ to mean that $\lambda$ is a
partition of $d$.     We  write
 $\lambda\trianglerighteq \mu$ and say that $\lambda$ dominates $\mu$  if
 $$
\textstyle \sum_{1\leq i \leq k}\lambda_i \geq  \sum_{1\leq i \leq k}\mu_i \; \text{  for all } k\geq 1.
 $$  

Given a weak composition $\lambda=(\lambda_1, \dots, \lambda_\ell)$ of
$d$, a Young diagram of shape $\lambda$ is a planar arrangement of
boxes into rows, with $\lambda_i$ boxes in the $i$th row, for each $i =
1, \dots, \ell$.  A $\lambda$-tableau $\stt$ is a numbering of the
boxes by the numbers $1, \dots, d$; i.e., a map from $\{1, \dots, d\}$
to the boxes.  A tableau is {\em row standard} if the numbers in each
row are increasing when read from left to right, and {\em standard} if
row standard and the numbers in each column are increasing when read
from top to bottom.    Given $1\leq k \leq n$, we let $\stt{\downarrow}_{\{1,\dots ,k\}}$ be the subtableau of $\stt$ whose  entries belong to the set
  $\{1,\dots,k\}$.
We write $\stt \trianglerighteq  \sts$ if $\stt{\downarrow}_{\{1,\dots ,k\}} \trianglerighteq  \sts{\downarrow}_{\{1,\dots ,k\}}$ for all $1\leq k \leq n$ and refer to this as the dominance order on standard $\lambda$-tableaux.

If $H$ is a subgroup of a finite group $G$ and $V$ a left or right
$\Bbbk H$-module, where $\Bbbk$ is a commutative ring, we respectively
have the left or right induced module
\[
\ind_H^G V = \Bbbk G \otimes_{\Bbbk H} V \text{ or }
V \otimes_{\Bbbk H} \Bbbk G.
\]
To each weak composition $\lambda =
(\lambda_1, \dots, \lambda_\ell)$ of $d$ there corresponds the
following data:
\begin{enumerate}\renewcommand{\labelenumi}{(\roman{enumi})}
\item The {\em row-reading tableau} $\stt^\lambda$ of shape $\lambda$, in
  which the numbers $1, \dots, d$ are written from left to right in the
  rows.
\item The Young subgroup $\Symset_\lambda$ of $\Symset_d$; this the subgroup
  of $\Symset_d$ stabilising the rows of $\stt^\lambda$.
\item The permutation module $M^\lambda$, defined by
\[
   M^\lambda = \ind_{\Symset_\lambda}^{\Symset_d} \Bbbk.
\]
It has a (tabloid) basis \cite{James} indexed by the set of
row-standard tableaux of shape $\lambda$.
\end{enumerate}
Permutation modules (both as left and right modules) play an important
role in what follows. As a left $\Bbbk \Symset_d$-module, it is well
known that $M^\lambda$ is isomorphic to the left ideal $\Bbbk
\Symset_d x_\lambda$, where $x_\lambda$ is defined in the next
paragraph.

Given a row-standard $\lambda$-tableau $\stt$, let $d(\stt)$ be the
unique element of $\Symset_d$ such that $\stt = d(\stt) \stt^\lambda$.
Let $\sgn w$ be the sign of a permutation $w$. Given any pair
$\sts,\stt$ of row-standard $\lambda$-tableaux, following Murphy, we
set
\begin{equation}\label{eq:Murphy-basis}
 x^\lambda_{\sts\stt} = d(\sts)x_\lambda d(\stt)^{-1} ; \quad 
 y^\lambda_{\sts\stt} = d(\sts) y_\lambda d(\stt)^{-1} .
\end{equation}
where $x_\lambda = \sum_{w \in W_\lambda} w$ and $y_\lambda = \sum_{w
  \in W_\lambda} (\sgn w)\,w$. If $\lambda$ is already specified in
context, then we may omit the superscript $\lambda$ from the notation,
writing $x_{\sts\stt}$, $y_{\sts\stt}$ instead of the more cumbersome
$x^\lambda_{\sts\stt}$, $y^\lambda_{\sts\stt}$.  Write $[\stt]$ for
the shape $\lambda$ of a tableau $\stt$. Then $x_{\sts\stt} =
x^\lambda_{\sts\stt}$ where $\lambda = [\sts] = [\stt]$, and similarly
for the $y_{\sts\stt}$.

Graham and Lehrer \cite{Graham-Lehrer} introduced cellular algebras in
order to axiomitise certain common features of certain classes of
finite dimensional algebras. A cellular algebra is an algebra with a
distinguished basis (the cellular basis) indexed by triples
$(\lambda,\sts,\stt)$ where $\lambda$ varies over a poset. The basis
yields canonical pairwise non-isomorphic cell modules
$\Delta(\lambda)$, one for each $\lambda$ in the indexing set.
For each $\lambda$, $\sts$ and $\stt$ belong to an index set in
bijection with a basis of $\Delta(\lambda)$. 

Murphy \cites{Murphy92, Murphy95} found two cellular bases of the
Iwahori--Hecke algebra associated to $\Bbbk \Symset_d$. By
specialising the deformation parameter to $1$, we obtain cellular
bases of $\Bbbk \Symset_d$ as follows.

\begin{thm}[Murphy] \label{thm:Murphy}
  Let $\Bbbk $ be a commutative ring. Each of the two disjoint unions
  \begin{align*}
  \mathcal{X} &= \textstyle\bigsqcup_{\lambda \vdash d}
  \{x_{\sts\stt}\mid \sts,\stt \text{ standard},\ [\sts] = [\stt] =
  \lambda \}, \\ \mathcal{Y} &= \textstyle\bigsqcup_{\lambda \vdash d}
  \{y_{\sts\stt} \mid \sts,\stt \text{ standard},\ [\sts] = [\stt] =
  \lambda \}
  \end{align*}
  is a cellular $\Bbbk$-basis of the group algebra $\Bbbk \Symset_d$,
 with respect to the anti-involution $*$. 
 \color{black}
  The cells are ordered by the
  dominance order $\dom$ (so $\le$ in \cite{Graham-Lehrer} must be
  replaced by $\dom$ here) with the least dominant partition $(1^d)$
  at the top and the most dominant partition $(d)$ at the bottom. The
  cell module $\Delta(\lambda)$ indexed by $\lambda$ is isomorphic to
  the dual Specht module $S_\lambda$ for the $x$-basis and the Specht
  module $S^{\lambda^\tr}$ for the $y$-basis, where $\lambda^\tr$ is
  the transpose of $\lambda$.
\end{thm}

Note that $x_{\sts\stt}$ and $y_{\sts\stt}$ are interchanged by the
$\Bbbk$-linear algebra involution $\#$ of $\Bbbk \Symset_d$ defined on
basis elements by $w \mapsto (\sgn w) w$, for $w \in \Symset_d$.
This involution converts results about one basis into results about
the other.

\begin{rmk}
We need to distinguish notationally between two symmetric groups in
this paper: $W_n \cong \Symset_n$ and $\Sym_r \cong \Symset_r$. We
write maps on the left in the former, on the right in the
latter. Thus, we compose from right-to-left in $W_n$ and from
left-to-right in $\Sym_r$. These groups act on $\bV^{\otimes r}$ on
the left and right by value and place-permutation, respectively. 
 We will make this   explicit  in the next section.   \color{black} 
Any
notation applicable to $\Symset_n$ will be extended to both $W_n$ and
$\Sym_r$; in particular we have the Young subgroups $W_\lambda$ and
$\Sym_\mu$ for any weak compositions $\lambda$, $\mu$ of $n$, $r$
respectively.  
\end{rmk}

\section{The $(\Bbbk W_n, \Ptn_r(n))$-bimodule $\bV^{\otimes r}$}%
\label{sec:bimod}\noindent
For the rest of the paper we fix a free $\Bbbk$-module $\bV$ of rank
$n$, with a given $\Bbbk$-basis $\{\bv_1, \dots, \bv_n\}$, where
$\Bbbk$ is a commutative ring.  We identify $\bV$ with $\Bbbk^n$ by
taking coordinates in the basis. For any positive integer $r$, the set
\begin{equation}\label{eq:TS-basis}
\{ \bv_{i_1} \otimes \cdots \otimes \bv_{i_r} \mid i_1,\dots, i_r = 1,
\dots, n \}
\end{equation}
is a basis of the $r$th tensor power $\bV^{\otimes r}$. The general
linear group $\GL(\bV)$ of $\Bbbk$-linear automorphisms of $\bV$ acts
naturally on the left on $\bV$; this action extends diagonally to an
action on $\bV^{\otimes r}$. The symmetric group $\Sym_r$ acts on the
right on $\bV^{\otimes r}$ by permuting the tensor positions; this
action is known as the \emph{place-permutation} action, defined by
\begin{equation}\label{eq:PP-action}
(\bv_{i_1} \otimes \cdots \otimes \bv_{i_r})^\sigma = \bv_{i_{1\sigma^{-1}}}
\otimes \cdots \otimes \bv_{i_{r \sigma^{-1}}} \, , \text{ for } \sigma \in
\Sym_r
\end{equation}
extended linearly. Note that we write maps in $\Sym_r$ on the
\emph{right} of their arguments, so that we may compose permutations
in $\Sym_r$ from left-to-right.
 The actions of the groups $\GL(\bV)$, $\Sym_r$ on $\bV^{\otimes r}$
commute, thus their linear extensions to the group algebras makes
$\bV^{\otimes r}$ into a $(\Bbbk \GL(\bV), \Bbbk \Sym_r)$-bimodule.

Basis elements $\bv_{i_1} \otimes \cdots \otimes \bv_{i_r}$ of
$\bV^{\otimes r}$ are indexed by multi-indices $(i_1, \dots, i_r)$ in
the set
\[
I(n,r) = \{1, \dots, n \}^r.
\]
In the sequel, we will be careful to distinguish between values
$i_\alpha$ in $\{1, \dots, n\}$ and places $\alpha$ in $\{1, \dots,
r\}$. For example, the multi-index $(2,7,7,6,2)$ takes the value $7$
in places $2,3$, the value $2$ in places $1,5$, and the value $6$ in
place $4$. In general, we will use Latin letters such as $i, j$ to
denote values and Greek letters such as $\alpha, \beta$ to denote
places.

Let $W_n$ be the Weyl group of $\GL(\bV)$, i.e., the group of elements
of $\GL(\bV)$ permuting the basis $\{\bv_1, \dots, \bv_n\}$.  We identify
$W_n$ with the group of permutation matrices, regarded as matrices
with entries from $\Bbbk$. By restricting the action of $\GL(\bV)$ to
$W_n$, we obtain a left action of $W_n$ on $\bV^{\otimes r}$. To be
explicit, $w \in W_n$ acts by
\begin{equation}\label{eq:W_n-action}
  w (\bv_{j_1} \otimes \cdots \otimes \bv_{j_r}) = \bv_{w(j_1)} \otimes
  \cdots \otimes \bv_{w(j_r)}.
\end{equation}
Note that we write maps on the \emph{left} of their arguments when
considering this action.
 Extended linearly, the action of $W_n$ defines a linear representation
$\Bbbk W_n \to \End_\Bbbk( \bV^{\otimes r} )$ of the group algebra
$\Bbbk W_n$.


For a positive integer $r$, and any $\delta\in\Bbbk$, we let
$\Ptn_r(\delta)$ denote the $\Bbbk$-module with basis given by all
set-partitions of $\{1,2,\ldots, r, 1',2', \ldots, r'\}$. By a
\emph{set-partition} we mean a pairwise disjoint covering of the
set.   An element  of a set-partition is called a \emph{block}.  For
example,
\[
d=\{\{1, 2, 4, 2', 5'\}, \{3\}, \{5, 6, 7, 3', 4', 6', 7'\}, \{8,
8'\}, \{1'\}\}
\]
is a set-partition of the set $\{1, \dots, 8, 1', \dots,
8' \}$ with five blocks.

We can depict each set-partition by a \emph{partition diagram},
consisting of $r$ northern nodes indexed by $1,2, \ldots ,r$ and $r$
southern nodes indexed by $1', 2', \ldots , r'$, with edges between
nodes, such that the nodes in the connected components give the blocks
of the set-partition. In general there are many partition diagrams
depicting a given set-partition, 
which we identify as equivalent.  

We define the product $x \cdot y$ of two diagrams $x$ and $y$ by
stacking $x$ above $y$, where we identify the southern nodes of $x$
with the northern nodes of $y$; these identified nodes then become the
middle nodes.  If there are $m$ connected middle components, then the
product $xy$ is set equal to $\delta^m$ times the diagram obtained 
by 
deleting the middle components (including middle vertices).  An
example is given in Figure \ref{fig:partmult}.
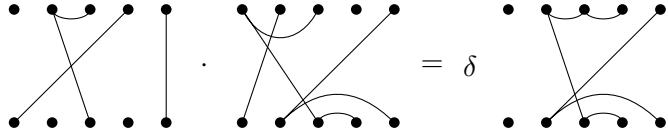
\begin{figure}[h!]
	\centering
	\begin{tikzpicture}[scale=0.5]
	 \foreach \x in {1,...,5,7,8,...,11,14,15,...,18}
	 	{\fill[black] (\x,0) circle (4pt) coordinate [below] (hi\x); 
	 	\fill[black] (\x,3) circle (4pt) coordinate [above] (2hi\x);}
		
		\draw(hi1) node[below]{\scriptsize $1'$};
	 	\draw(2hi1) node[above] {\scriptsize $1$};
		\draw(hi2) node[below] {\scriptsize $2'$};
	 	\draw(2hi2) node[above] {\scriptsize $2$};
	\draw(hi3) node[below] {\scriptsize $3'$};
	 	\draw(2hi3) node[above] {\scriptsize $3$};
	\draw(hi4) node[below] {\scriptsize $4'$};
	 	\draw(2hi4) node[above] {\scriptsize $4$};
	\draw(hi5) node[below] {\scriptsize $5'$};
	 	\draw(2hi5) node[above] {\scriptsize $5$};

		\draw(hi7) node[below] {\scriptsize $1'$};
	 	\draw(2hi7) node[above] {\scriptsize $1$};
		\draw(hi8) node[below] {\scriptsize $2'$};
	 	\draw(2hi8) node[above] {\scriptsize $2$};
	\draw(hi9) node[below] {\scriptsize $3'$};
	 	\draw(2hi9) node[above] {\scriptsize $3$};
	\draw(hi10) node[below] {\scriptsize $4'$};
	 	\draw(2hi10) node[above] {\scriptsize $4$};
	\draw(hi11) node[below] {\scriptsize $5'$};
	 	\draw(2hi11) node[above] {\scriptsize $5$};

		\draw(hi14) node[below] {\scriptsize $1'$};
	 	\draw(2hi14) node[above] {\scriptsize $1$};
		\draw(hi15) node[below] {\scriptsize $2'$};
	 	\draw(2hi15) node[above] {\scriptsize $2$};
	\draw(hi16) node[below] {\scriptsize $3'$};
	 	\draw(2hi16) node[above] {\scriptsize $3$};
	\draw(hi17) node[below] {\scriptsize $4'$};
	 	\draw(2hi17) node[above] {\scriptsize $4$};
	\draw(hi18) node[below] {\scriptsize $5'$};
	 	\draw(2hi18) node[above] {\scriptsize $5$};

	\draw (3,3) arc (0:-180:0.5 and 0.25) -- (3,0);\draw (4,3)--(1,0);\draw (5,3)--(5,0);
	\draw (6,1.5) node {$\cdot$};
	\draw (9,3) .. controls (8.5,2) and (7.5,2) .. (7,3) -- (9,0) arc (180:0:0.5 and 0.25);
	\draw (8,3)--(7,0);
	\draw (11,0) .. controls (10,1) and (9,1) .. (8,0) -- (11,3);
	\draw (12,1.5) node {$=$};\draw (13,1.5) node {$\delta$};
	\draw (18,0) .. controls (17,1) and (16,1) .. (15,0) -- (18,3);
	\draw (17,3) arc (0:-180:0.5 and 0.25) arc (0:-180:0.5 and 0.25) -- (16,0) arc (180:0:0.5 and 0.25);
	\end{tikzpicture}
	\caption{Multiplication of two diagrams in $\Ptn_5 (\delta)$.}
\label{fig:partmult}
\end{figure}
Extending the product linearly defines a multiplication on
$\Ptn_r(\delta)$, making it an associative algebra. Note that $\Bbbk
\Sym_r \subset \Ptn_r(\delta)$ is the subalgebra spanned by the
\emph{permutation diagrams}, the diagrams depicting set-partitions
with $r$ blocks, each of which contains exactly one element of $\{1,
\dots, r\}$ and one of $\{1', \dots, r'\}$.

To obtain an action of the partition algebra on $\bV^{\otimes r}$ it
is necessary to specialise $\delta$ to $n = \rank_\Bbbk
\bV$. Following \cite{Halverson-Ram}, we define a generalised
Kronecker delta symbol $(d)^{i_1, \dots, i_r}_{i_{1'}, \dots, i_{r'}}$
corresponding to a diagram $d$ and any $(i_1, \dots, i_r)$, $(i_{1'},
\dots, i_{r'})$ in $I(n,r)$, by
\[
(d)^{i_1, \dots, i_r}_{i_{1'}, \dots, i_{r'}} = 
\begin{cases}
  1 & \text{if } i_\alpha = i_\beta \text{ whenever } \alpha \ne
  \beta \text{ are in the same block of } d\\ 0 & \text{otherwise}.
\end{cases}
\]
Then the diagram $d$ acts on $\bV^{\otimes r}$, on the right, by
the rule
\begin{equation}\label{eq:Ptn-action}
(\bv_{i_1} \otimes \cdots \otimes \bv_{i_r})^d = \sum_{(i_{1'}, \dots,
    i_{r'}) \in I(n,r)} (d)^{i_1, \dots, i_r}_{i_{1'}, \dots,
    i_{r'}} \, (\bv_{i_{1'}} \otimes \cdots \otimes \bv_{i_{r'}}).
\end{equation}
Extended linearly, this action defines a linear representation
$\Ptn_r(n)^{\op} \to \End_\Bbbk(\bV^{\otimes r})$.  If $d \in \Sym_r$
is a permutation diagram, then $d$ acts by the place-permutation
action defined in \eqref{eq:PP-action}. The actions of $W_n$ and
$\Ptn_r(n)$ defined in \eqref{eq:W_n-action} and \eqref{eq:Ptn-action}
commute, thus we have a $(\Bbbk W_n, \Ptn_r(n))$-bimodule structure on
$\bV^{\otimes r}$.

\section{The $(\Bbbk W_{n-1}, \Ptn_{r+\,1/2}(n))$-bimodule $\bV^{\otimes r}$}%
\label{sec:bimod-restr}\noindent
Let $\Ptn_{r+\,1/2}(\delta)$ denote the submodule of
$\Ptn_{r+1}(\delta)$ with $\Bbbk$-basis given by all set-partitions
such that $r+1$ and $(r+1)'$ belong to the same block.  The submodule
$\Ptn_{r+\,1/2}(\delta)$ is closed under the multiplication and
therefore is a subalgebra of $\Ptn_{r+1}(\delta)$.  The
$\Bbbk$-submodule $\bV^{\otimes r}\otimes \bv_n \subset \bV^{\otimes
  (r+1) }$ is stable under the action of
$\Ptn_{r+\,1/2}(n)$. Therefore, by identifying $\bV^{\otimes r}$ with
\[
\bV^{\otimes r}\otimes \bv_n \subset \bV^{\otimes (r+1)}
\]
we regard $\bV^{\otimes r}$ as a right $\Ptn_{r+\,1/2}(n)$-module. We
also regard it as a left $\Bbbk W_{n-1}$-module by restriction from
$\Bbbk W_n$, where
\[
W_{n-1} = \{w \in W_n\mid w(n) = n \}.
\]
Thus, after identifying $\bV^{\otimes r}$ with $\bV^{\otimes r} \otimes
\bv_n$, we have a $(\Bbbk W_{n-1}, \Ptn_{r+\,1/2}(n))$-bimodule
structure on $\bV^{\otimes r}$.

\section{Decompositions of $\bV^{\otimes r}$}
\noindent
Henceforth we study $\bV^{\otimes r}$ as left $\Bbbk W_n$-module and
also as left $\Bbbk W_{n-1}$-module (subject to the identification of
$\bV^{\otimes r}$ with $\bV^{\otimes r} \otimes \bv_n$, discussed in
Section \ref{sec:bimod-restr}).
We write
\[
\Phi_{n,r} : \Bbbk W_n \to \End_{\Bbbk}(\bV^{\otimes r}), \quad
\Phi_{n,r+\,1/2} : \Bbbk W_{n-1} \to \End_{\Bbbk}(\bV^{\otimes r})
\]
for the $\Bbbk$-linear representations corresponding to the left
actions in the two bimodule structures on $\bV^{\otimes r}$.  Our goal
is to understand the annihilator of each action. In this section we
obtain direct sum decompositions of $\bV^{\otimes r}$, as both left
$\Bbbk W_n$ and $\Bbbk W_{n-1}$-modules, but first we record the
following elementary fact.

\begin{lem}\label{lem:faithful}
  If $r \ge n$ then the representation $\Phi_{n,r}: \Bbbk W_n
  \to \End_\Bbbk(\bV^{\otimes r})$ is faithful. If $r+1 \ge n$ then
  the representation $\Phi_{n,r+\, 1/2}: \Bbbk W_{n-1}
  \to \End_\Bbbk(\bV^{\otimes r} \otimes \bv_n)$ is faithful.
\end{lem}

\begin{proof}
Suppose that $a = \sum_{w \in W_n} a_w \, w$ belongs to the kernel of
$\Phi_{n,r}$. Then $a$ acts as zero on all elements of $\bV^{\otimes
  r}$. First suppose that $r = n$. Consider the simple tensor $v =
\bv_1 \otimes \bv_2 \otimes \cdots \otimes \bv_n$. We have
\[
w \cdot v = \bv_{w(1)} \otimes \bv_{w(2)} \otimes \cdots \otimes \bv_{w(n)}
\]
and each $w \cdot v$ is a simple tensor obtained from $v$ by
permuting its factors according to $w$. In particular, the set $\{ w
\cdot v\mid w \in W_n\}$ is linearly independent over $\Bbbk$. Thus the
fact that $a \cdot v = 0$ implies that
\[
\textstyle \sum_{w \in W_n} a_w \, (\bv_{w(1)} \otimes \bv_{w(2)} \otimes
\cdots \otimes \bv_{w(n)}) = 0.
\]
By linear independence, this forces $a_w = 0$ for all $w \in W_n$;
that is, $a = 0$. This proves the first claim in case $r=n$.

For the general case, $r \ge n$, replace $v$ by $v \otimes
\bv_n^{\otimes (r-n)}$ and repeat the argument. This proves the first
claim. The proof of the second claim is similar.
\end{proof}

\begin{rmk}\label{rmk:sharper}
Lemma \ref{lem:faithful} will be sharpened in Corollary
\ref{cor:faithful}. The preceding argument can be modified to prove
the sharpened result; we leave the details to the interested reader.
\end{rmk}

Although not needed in this paper, for purposes of comparison we
recall the standard multiplicity-free decomposition of $\bV^{\otimes
  r}$ as a right $\Bbbk \Sym_r$-module:
\begin{equation}
  \textstyle \bV^{\otimes r} \cong \bigoplus_{\lambda} M^\lambda
\end{equation}
where the sum is over all weak compositions of $r$ of length at most
$n$. To see this, observe that each $M^\lambda$ may be identified with
the weight space $\bV^{\otimes r}_\lambda$ consisting of all tensors
of weight $\lambda$ for the action of the diagonal torus $T \subset
\GL(\bV)$ (elements of $\GL(\bV)$ acting diagonally on the basis
vectors $\bv_i$). The identification $\bV^{\otimes r}_\lambda \cong
M^\lambda$ is given on basis elements by
\begin{equation}\label{eq:wt-space-ident}
\bv_{i_1} \otimes \cdots \otimes \bv_{i_r} \mapsto \stt(i_1, \dots, i_r),
\end{equation}
where $\stt(i_1, \dots, i_r)$ is the row-standard $\lambda$-tableau
whose $j$th row contains all the tensor places in which $\bv_j$ appears.


 We now provide a language for  decomposing $\bV^{\otimes r}$, both as a
left $\Bbbk W_n$ and $\Bbbk W_{n-1}$-module.  
The direct summands of this decomposition will be labelled by {\em hook partitions}, and this will be   important in what follows.  
 The combinatorics used  here is completely different to that of  the characteristic zero $(\Bbbk W_{n-1}, \Ptn_{r+\,1/2}(n))$-bimodule decomposition, instead it uses the language from Section \ref{sec:comb} (and ultimately goes back to ideas from \cite{James}).  
 
\color{black}

\begin{defn}\label{def:value-type}
  The \emph{value-type} of a multi-index $(i_1, \dots, i_r)$ in
  $I(n,r)$ is the set-partition $\Lambda$ of $\{1, \dots, r\}$ defined
  by $\Lambda = \{ \Lambda_1, \dots, \Lambda_n \}$, where $\Lambda_j$
  is the set of places $\alpha = 1, \dots r$ such that $i_\alpha =
  j$. By convention, we usually omit any empty subsets $\Lambda_j$
  from $\Lambda$. The non-empty subsets in $\Lambda$ are called
  \emph{parts}; their number is denoted by $\ell(\Lambda)$ and is
  called the \emph{length} of $\Lambda$.
\end{defn}

The non-empty parts of the value-type $\Lambda$ associated to $(i_1,
\dots, i_r)$ are the subsets given by the non-empty rows of the
associated tableau $\stt(i_1, \dots, i_r)$ defined in
\eqref{eq:wt-space-ident}. For instance, the value-type of the
multi-index
\[
(i_1, \dots, i_7) = (9,8,8,1,9,8,1)
\]
is $\Lambda = \{\{1,5\}, \{2,3,6\}, \{4,7\}\}$, a set-partition of
$\{1, \dots, 7\}$ with three parts.

We remark that value-type may be visualised as a half-diagram, in which nodes are linked by edges if and only if they have the same value in the multi-index (although we do not use this visualisation in the paper).  

\begin{defn}\label{def:M-and-N}
  Let $\Lambda$ be a set-partition of $\{1, \dots, r\}$ with not more
  than $n$ parts. We define $\bV(\Lambda)$ to be the $\Bbbk$-span of the
  tensors $\bv_{i_1} \otimes \cdots \otimes \bv_{i_r}$ such that the
  value-type of $(i_1, \dots, i_r)$ is equal to $\Lambda$.

  Similarly, for each set-partition $\Lambda'$ of $\{1, \dots, r+1\}$
  with not more than $n$ parts, we define $\bV'(\Lambda')$ to be the
  $\Bbbk$-span of the tensors $\bv_{i_1} \otimes \cdots \otimes \bv_{i_r}
  \otimes \bv_n$ such that the value-type of $(i_1, \dots, i_{r}, n)$ is
  equal to $\Lambda'$.
\end{defn}

Both $\bV(\Lambda)$, $\bV'(\Lambda')$ are free modules over the
commutative ring $\Bbbk$.  It is useful to understand how
$\bV'(\Lambda')$ is related to $\bV(\Lambda')$. Of course we have
$\bV'(\Lambda') \subset \bV(\Lambda')$ by definition. To be more
specific, we have the following result.

\begin{lem}\label{lem:M-and-N}
  For any set-partition $\Lambda'$ of $\{1, \dots, r+1\}$ with not
  more than $n$ parts, $\rank_\Bbbk \bV'(\Lambda') = \frac{1}{n} \,
  \rank_\Bbbk \bV(\Lambda')$.
\end{lem}

\begin{proof}
Consider a simple tensor of the form $v = \bv_{i_1} \otimes \cdots
\otimes \bv_{i_r} \otimes \bv_n$ in $\bV'(\Lambda')$. Let $\Lambda'_n$
be the part of $\Lambda'$ recording the places in $v$ at which $\bv_n$
appears (so of course $r+1 \in \Lambda'_n$). For any $j = 1, \dots,
n-1$ the tensor
\[
s \cdot v = s \cdot (\bv_{i_1} \otimes \cdots \otimes \bv_{i_r} \otimes
\bv_n) 
\]
belongs to $\bV(\Lambda')$, where $s \in W_n$ is the transposition $s =
(j,n)$. For each $j=1, \dots, n-1$, the map $f_j$ from $\bV'(\Lambda')$
into $\bV(\Lambda')$ defined on simple tensors by $v \mapsto s(v)$, then
extended linearly, is injective. Furthermore, the image of each $f_j$
gives an isomorphic copy of $\bV'(\Lambda')$ inside
$\bV(\Lambda')$. Finally, the embedding $f_n$ of $\bV'(\Lambda')$ in
$\bV(\Lambda')$ defined by $v \mapsto v$ gives another isomorphic copy
of $\bV'(\Lambda')$ in $\bV(\Lambda')$. Since $\bV(\Lambda')$ is equal to
the direct sum of the images of the $n$ maps $f_1, \dots, f_n$, it
follows that $\rank_\Bbbk \bV(\Lambda') = n \cdot \rank_\Bbbk
\bV'(\Lambda')$.
\end{proof}

\begin{rmk}
It is straightforward  to check that $\bV(\Lambda') \cong \ind_{W_{n-1}}^{W_n}
\bV'(\Lambda')$, the module obtained by inducing $\bV'(\Lambda')$ from
$W_{n-1}$ to $W_n$.
\end{rmk}
Note that, by definition, simple tensors in $\bV'(\Lambda')$ always have
the fixed vector $\bv_n$ appearing in place $r+1$. We also note that
whenever $\Lambda$ has more than $n$ parts there are no simple tensors
in $\bV^{\otimes r}$ of value-type $\Lambda$, so $\bV(\Lambda) = 0$. On
the other hand, it is easy to see that $\bV(\Lambda) \ne 0$ if
$\ell(\Lambda) \le n$. It follows from Lemma \ref{lem:M-and-N} that
$\bV'(\Lambda') \ne 0$ if and only if $\ell(\Lambda') \le n$.

Here then are the promised decompositions.

\begin{prop}\label{prop:decomposition}
  The action of $W_n$ on $\bV^{\otimes r}$ preserves value-type of
  simple tensors, so $\bV(\Lambda)$ is a left $\Bbbk
  W_n$-module. Similarly, $\bV'(\Lambda')$ is a left $\Bbbk
  W_{n-1}$-module. Furthermore, we have direct sum decompositions
  \begin{equation*}
    \textstyle \bV^{\otimes r} = \bigoplus_{\ell(\Lambda) \le n}
    \bV(\Lambda), \quad \bV^{\otimes r} \otimes \bv_n =
    \bigoplus_{\ell(\Lambda') \le n} \bV'(\Lambda')
  \end{equation*}
  where $\Lambda$, $\Lambda'$ vary over all set-partitions (with not
  more than $n$ parts) of $\{1, \dots, r\}$, $\{1, \dots, r+1\}$,
  respectively. The decompositions are multiplicity-free, in the sense
  that each $\bV(\Lambda)$, $\bV'(\Lambda')$ appears exactly once in the
  direct sum.
\end{prop}

\begin{proof}
All the claims are easily verified. The point is that the
classification of simple tensors by value-type describes the orbits of
simple tensors under the left action of $W_n$. The claims for
$W_{n-1}$ are just variations on this theme.
\end{proof}

The problem with the decompositions in Proposition
\ref{prop:decomposition} is that the summands are not pairwise
non-isomorphic (as $\Bbbk W_n$ or $\Bbbk W_{n-1}$-modules). By looking
at examples one quickly observes that $\bV(\Lambda)$ depends, up to
isomorphism, only on the number of parts $\ell(\Lambda)$ and not on
$\Lambda$ itself.  Similar remarks apply to the $\bV'(\Lambda')$ in
the second decomposition.

To overcome this difficulty, we introduce minimal prototypes for the
isomorphism classes that can occur as summands in the above
decompositions of tensor space.

\begin{defn}
For any $l = 1, \dots, n$ let $H_n(l)$ be the $\Bbbk W_n$-submodule of
$\bV^{\otimes l}$ generated by $\bv_{n-l+1}\otimes \bv_{n - l + 2} \otimes \cdots \otimes
\bv_n$. 
\end{defn}

Here are some elementary properties of the $H_n(l)$.

\begin{lem}\label{lem:H-isos}
As left $\Bbbk W_n$-modules, we have isomorphisms:
\begin{enumerate}\renewcommand{\labelenumi}{(\alph{enumi})}
\item $H_n(l) \cong M^{(n-l, 1^l)}$, for any $l = 1, \dots, n-1$.
\item $H_n(n) \cong M^{(1^n)} \cong \Bbbk W_n$.
\item $H_n(n-1) \cong H_n(n)$.
\end{enumerate}
\end{lem}

\begin{proof}
As (c) follows from (a), (b) we only need to prove (a), (b). To prove
(a), observe that $H_n(l)$ is a transitive permutation module, because
its basis elements $\{ w \cdot (\bv_{n-l+1} \otimes \cdots \otimes
\bv_{n}) \mid w \in W_n\}$ are permuted transitively by the action of
$W_n$. The stabiliser of $\bv_{n-l+1} \otimes \cdots \otimes \bv_{n}$ is
the Young subgroup $W_{n-l} = \{w \in W_n \mid w(k) = k \text{ for
  all } k = n-l+1, \dots, n\}$. Since
\[
W_{n-l} \cong W_{\{1, \dots, n-l\}} \times W_{\{n-l+1\}} \times
\cdots \times W_{\{n\}}
\]
this is the Young subgroup indexed by the partition $(n-l,
1^l)$. Hence $H_n(l) \cong \ind_{W_{n-l}}^{W_n} \Bbbk \cong M^{(n-l,1^l)}$, and
(a) is proved.  For (b), observe that the stabiliser of the generator
$\bv_1 \otimes \cdots \otimes \bv_n$ is the trivial subgroup, which is the
Young subgroup indexed by $(1^n)$. For another way to prove (c),
observe that the stabiliser of $\bv_2 \otimes \cdots \otimes \bv_n$ is
also trivial, because any permutation in $W_n$ fixing $n-1$ points
must fix all $n$ points.
\end{proof}

\begin{prop}\label{prop:perm-modules}
  Let $\Lambda$, $\Lambda'$ be set-partitions of $\{1, \dots, r\}$,
  $\{1, \dots, r+1\}$ respectively, with not more than $n$
  parts. Then:
  \begin{enumerate}\renewcommand{\labelenumi}{(\alph{enumi})}
  \item $\bV(\Lambda) \cong H_n(\ell(\Lambda))$, as left $\Bbbk
    W_n$-modules.
  \item $\bV'(\Lambda') \cong H_{n-1}(\ell(\Lambda')-1)$, as left
    $\Bbbk W_{n-1}$-modules.
  \end{enumerate}
\end{prop}

\begin{proof}
(a) The isomorphism is given by sending each $\bv_{i_1} \otimes \cdots
  \otimes \bv_{i_l}$ in $H_n(l)$ to the simple tensor in $\bV(\Lambda)$
  obtained by writing $\bv_{i_k}$ into all tensor places in $\Lambda_k$,
  where $\Lambda = \{\Lambda_1, \dots, \Lambda_l\}$.

(b) The proof of (b) is similar to the proof of (a), except that $\bv_n$
  is fixed wherever it appears. Write $\Lambda' = \{\Lambda_1, \dots,
  \Lambda_l\}$ and assume that $r+1$ belongs to $\Lambda_l$, so
  $\Lambda_l$ records the places containing a $\bv_n$. Then the
  isomorphism is defined by sending $\bv_{i_1} \otimes \cdots \otimes
  \bv_{i_{l-1}}$ in $H_{n-1}(l-1)$ to the simple tensor in
  $\bV'(\Lambda')$ obtained by writing $\bv_{i_k}$ into all tensor
  places in $\Lambda_k$, for $k=1, \dots, l-1$, and writing $\bv_n$ in
  the places in $\Lambda_l$.
\end{proof}

\begin{cor}\label{cor:decomps}
For any commutative ring $\Bbbk$, we have isomorphisms
\[ 
\bV^{\otimes r} \cong \bigoplus_{1 \le l \le \min(n,r)}
H_n(l)^{\stirling{r}{l}}, \quad \bV^{\otimes r} \otimes \bv_n
\cong \bigoplus_{1 \le l \le \min(n-1,r)}
H_{n-1}(l)^{\stirling{r+1}{l}}
\]
as left $\Bbbk W_n$-modules, $\Bbbk W_{n-1}$-modules,
respectively. The Stirling numbers $\stirling{r}{l}$,
$\stirling{r+1}{l}$ give the multiplicities of the direct summands in
the decompositions.
\end{cor}

\begin{proof}
By Proposition \ref{prop:perm-modules} the direct summands in
Proposition \ref{prop:decomposition} are isomorphic to $H_n(l)$,
$H_{n-1}(l)$, where $1 \le l \le \min(n,r)$, $\min(n-1,r)$
respectively.  The number of set partitions $\Lambda$ of $\{1, \dots,
r\}$ for which $\ell(\Lambda) = l$ is given by $\stirling{r}{l}$, the
Stirling number of the second kind \cite{Stanley}*{\S1.9}. It follows
that the number of set partitions $\Lambda'$ of $\{1, \dots, r+1\}$
for which $\ell(\Lambda') = l$ is given by $\stirling{r+1}{l}$.
\end{proof}

\begin{rmk}
If $n \le r$, $n-1 \le r$, respectively, then the direct summands in
the decompositions in Corollary \ref{cor:decomps} are still not
pairwise non-isomorphic. In those cases, we have to take the
isomorphism in Lemma \ref{lem:H-isos}(c) into account. This 
is worth living with in order to have a uniform formula for the
multiplicities of the summands. 
\end{rmk}

As another consequence of these results, we easily obtain the promised
sharpening (see Remark \ref{rmk:sharper}) of Lemma \ref{lem:faithful}.

\begin{cor}\label{cor:faithful}
  For any commutative ring $\Bbbk$, the representations $\Phi_{n,r}$,
  $\Phi_{n,r+\, 1/2}$ are faithful when $n-1 \le r$, $n-2 \le r$,
  respectively.
\end{cor}

\begin{proof}
This follows immediately from the decompositions in Corollary
\ref{cor:decomps} and the isomorphisms in Lemma \ref{lem:H-isos}(b),
(c) which imply that $H_n(n-1) \cong \Bbbk W_n$, $H_{n-1}(n-2) \cong
\Bbbk W_{n-1}$ are faithful modules.
\end{proof}

\begin{rmk}
   Equation \eqref{eq:when-zero} in Section \ref{sec:ann-gen} implies
  that $\Phi_{n,r}$, $\Phi_{n,r+1/2}$ are not faithful when $n-1 > r$,
  $n-2>r$ respectively, so the bounds in Corollary \ref{cor:faithful}
  are best possible.
\end{rmk}

\section{The annihilator of $\bV^{\otimes r}$ in characteristic zero}
\label{sec:ss}\noindent
Recall that we always identify $\bV^{\otimes r}$ with $\bV^{\otimes r}
\otimes \bv_n$ when considering the representation $\Phi_{n,r+\, 1/2}:
\Bbbk W_{n-1} \to \End_\Bbbk(\bV^{\otimes r})$. We also have the
representation $\Phi_{n,r}: \Bbbk W_n \to \End_\Bbbk(\bV^{\otimes
  r})$. From now on we treat the two representations uniformly,
writing them as
\begin{equation}\label{eq:the-reps}
\Phi_{n,r+\eps}: \Bbbk W_d \to \End_\Bbbk(\bV^{\otimes r}),
\end{equation}
where $\eps \in \{0, 1/2\}$ and $d=d(n,\eps):=n-2\eps$. Henceforth,
the symbols $\eps$, $d$ will always have these fixed interpretations.
We wish to study $\ker \Phi_{n,r+\eps}$, so from now on we always
assume that $d > r+1$, because otherwise $\Phi_{n,r+\eps}$ is
faithful, by Corollary \ref{cor:faithful}.

In this section we assume that $\Bbbk$ is a field of characteristic
zero. This implies that $\Bbbk W_d$ is semisimple.  By the
Artin--Wedderburn theorem and the fact that all irreducible
representations of $W_d$ are absolutely irreducible, we have a (split)
semisimple decomposition
\begin{equation}\label{eq:ss}
  \Bbbk W_d \cong \textstyle \bigoplus_{\mu\, \vdash\, d} \End_\Bbbk(S^\mu)
\end{equation}
where $S^\mu$ is the (irreducible) Specht module indexed by $\mu
\vdash d$. The assumption $d>r+1$ implies that $\min(d,r) = r$, so we
may combine the two cases of Corollary \ref{cor:decomps} as:
$\bV^{\otimes r} \cong \bigoplus_{1 \le l \le r}
H_d(l)^{\stirling{r+2\eps}{l}}$. In light of the isomorphism from
Lemma \ref{lem:H-isos}(a), this gives the decomposition
\begin{equation}\label{eq:M-decomp}
  \bV^{\otimes r} \cong \textstyle \bigoplus_{\lambda \in \Hook(d,r)}
  (M^\lambda)^{m_\lambda} \quad (m_\lambda >0)
\end{equation}
as left $\Bbbk W_d$-modules, where $\Hook(d,r) = \{ (d-l, 1^l) \mid l
= 1, \dots, r\}$. Note that $\Hook(d,r)$ is a set of hook partitions
and the $M^\lambda$ such that $\lambda \in \Hook(d,r)$ are pairwise
non-isomorphic. The multiplicities $m_\lambda$ in \eqref{eq:M-decomp}
are given by Stirling numbers, but we only need that they are positive
integers.

\begin{defn}\label{def:alpha(d,r)}
  Write $\alpha(d,r) = (d-r, 1^r)$ for the minimum element (with
  respect to the dominance order, $\rhd$, defined in Section \ref{sec:comb})
  \color{black} of the set $\Hook(d,r)$.
\end{defn}       

If $S$ is a simple module and $M$ is a module satisfying the
Jordan--H\"{o}lder theorem, write $[M:S]$ for the multiplicity of $S$
in a composition series of $M$.  Our next result provides a lower
bound on $\ker \Phi_{n,r+\eps}$ in characteristic zero.

\begin{prop}\label{prop:kernel0-incl}
Assume that $\Bbbk$ is a field of characteristic zero. Let $d =
n-2\eps$ and assume that $d > r+1$, where $\eps \in \{0, 1/2\}$.  The
set of $\lambda \vdash d$ such that $[\bV^{\otimes r}: S^\lambda] \ne
0$ is contained in $\{ \lambda \vdash d \mid \lambda \dom \alpha(d,r)
\}$.  Hence, $\ker \Phi_{n,r+\eps}$ contains an isomorphic copy of
$\textstyle \bigoplus_{\lambda \vdash d, \,\lambda \notdom
  \alpha(d,r)} \End_\Bbbk(S^\lambda)$.
\end{prop}

\begin{proof}
The indexing set $\Hook(d,r)$ in the decomposition \eqref{eq:M-decomp}
forms a well-ordered chain under the dominance order. By Young's rule
(see for instance \cite{James}*{4.13 or 14.1}), for any $\mu \vdash
d$, the set of $\lambda \vdash d$ such that $[M^\mu: S^\lambda] \ne 0$
is contained in $\{ \lambda \vdash d \mid \lambda \dom \mu \}$. The
first claim follows, since
\[
\{ \lambda \dom \mu \mid \mu \in \Hook(d,r) \} =
\{ \lambda \vdash d \mid \lambda \dom \alpha(d,r) \}
\]
because $\alpha(d,r)$ is the minimum element of $\Hook(d,r)$.  Hence,
as a $\Bbbk W_d$--module, the decomposition \eqref{eq:M-decomp} takes
the form
\[
\bV^{\otimes r} \cong \textstyle \bigoplus_{\lambda \dom \alpha(d,r)}
(S^\lambda)^{n_\lambda},
\]
where $n_\lambda \ge 0$ is the multiplicity of $S^\lambda$ in the
decomposition.  The semisimple decomposition $\Bbbk W_d \cong
\bigoplus_{\lambda \vdash d} \End_\Bbbk(S^\lambda)$ then implies the
final statement in the proposition, since the only summands acting
non-trivially on $\bV^{\otimes r}$ are the $\End_\Bbbk(S^\lambda)$
such that $\lambda \dom \alpha(d,r)$ and $n_\lambda > 0$.
\end{proof}

The following fact from \cite{Doty-Nyman} can be applied to obtain the
opposite inclusion and thus prove equality in Proposition
\ref{prop:kernel0-incl}.

\begin{lem}[\cite{Doty-Nyman}*{Lemma 6.4}]
\label{lem:embedding-ss-case}
  Assume that $\Bbbk$ is a field of characteristic zero. For any
  partitions $\lambda \dom \mu$ there is an embedding of $M^\lambda$
  in $M^\mu$, as $\Bbbk W_d$-modules.
\end{lem}

If $M^\lambda$ embeds in $M^\mu$ then clearly $\ann_{\Bbbk W_d}
M^\lambda \supseteq \ann_{\Bbbk W_d} M^\mu$.

\begin{prop}\label{prop:kernel0-eq}
 Assume that $\Bbbk$ is a field of characteristic zero. Let $d =
 n-2\eps$ and assume that $d > r+1$, where $\eps \in \{0,
 1/2\}$. Then:

  (a) $\ker \Phi_{n,r+\eps}$ is isomorphic to $\textstyle
  \bigoplus_{\lambda \vdash d,\, \lambda \notdom \alpha(d,r)} \End_\Bbbk(S^\lambda)$.

  (b) $\ker \Phi_{n,r+\eps} = \ann_{\Bbbk W_d} M^{\alpha(d,r)} = \bigcap_{\lambda
    \vdash d, \, \lambda \dom \alpha(d,r)} \ann_{\Bbbk W_d}
  M^\lambda$.
\end{prop}

\begin{proof}
(a) Lemma \ref{lem:embedding-ss-case} implies that we do not change
  the annihilator by adding extra permutation module terms in the
  $\Bbbk W_d$-module decomposition \eqref{eq:M-decomp}, enlarging it
  to $\textstyle \bigoplus_{\mu \vdash d, \,\mu \dom \alpha(d,r)}
  (M^\mu)^{m_\mu}$, where $m_\mu := 1$ for any $\mu \notin
  \Hook(d,r)$.  But then any $S^\lambda$ for $\lambda \vdash d$ such
  that $\lambda \dom \alpha(d,r)$ must occur with multiplicity at
  least one, since $[M^\lambda: S^\lambda] = 1$. This implies that the
  annihilator cannot be larger than $\textstyle \bigoplus_{\lambda
    \vdash d,\, \lambda \notdom
    \alpha(d,r)} \End_\Bbbk(S^\lambda)$. The equality in part (a) then
  follows from Proposition \ref{prop:kernel0-incl}.

(b) By Lemma \ref{lem:embedding-ss-case}, for each $\lambda \dom
  \alpha(d,r)$ we have an embedding $M^\lambda \subseteq
  M^{\alpha(d,r)}$. It follows that there is an embedding $\ann_{\Bbbk
    W_d} M^\lambda \supseteq \ann_{\Bbbk W_d} M^{\alpha(d,r)}$. Hence
  the intersection
  $$\ann_{\Bbbk W_d} \bV^{\otimes r} = \bigcap_{\lambda \vdash d, \,
    \lambda \dom \alpha(d,r)} \ann_{\Bbbk W_d} M^\lambda$$ collapses
  to the single term $\ann_{\Bbbk W_d} M^{\alpha(d,r)}$. This proves
  part (b).
\end{proof}

\section{Reformulation of the characteristic zero result}\noindent
As in Section \ref{sec:ss}, we still assume that $\Bbbk$ is a field of
characteristic zero. Our standing assumption that $d > r+1$ remains in
force. Finally, we remind the reader that $d=n-2\eps$, where $\eps \in
\{0, 1/2\}$.

It is straightforward  to see by looking at examples that the $y$-basis
$\mathcal{Y}$  (of Theorem~\ref{thm:Murphy})
\color{black} is the right one to use in order to describe the kernel
of $\Phi_{n,r+\eps}$ as a cell ideal. For instance, if $n=3$ and $r=1$
then it is easy to check that the map $\Phi_{3,1}: \Bbbk W_3
\to \End(\bV)$ has kernel generated by the element
\[
  y_{\stt^{(3)}\stt^{(3)}} = \textstyle \sum_{w \in W_3} \sgn(w)\, w.  
\]
Note that this example is compatible with Proposition
\ref{prop:kernel0-eq}(a), which says that the kernel is isomorphic to
$\End_\Bbbk(S^{(1^3)}) = \End_\Bbbk(\Delta(3))$.

From now on we make heavy use of the fact that transposing reverses
the dominance order: $\lambda \dom \mu$ if and only if $\lambda^\tr
\lessdom \mu^\tr$.    
We now reformulate Proposition \ref{prop:kernel0-eq}
by replacing $\lambda$ by its transpose $\lambda^\tr$, in light of the
identification $\Delta(\lambda) = S^{\lambda^\tr}$ from Theorem
\ref{thm:Murphy}.  Then Proposition \ref{prop:kernel0-eq}(a) takes the
form
\begin{equation}
  \ker \Phi_{n,r+\eps} \cong \textstyle \bigoplus_{\lambda^\tr \notdom
    \alpha(d,r)} \End_\Bbbk(S^{\lambda^\tr}) = \textstyle
  \bigoplus_{\lambda \notlessdom \alpha(d,r)^\tr} \End_\Bbbk(\Delta(\lambda))
\end{equation}
where the rightmost equality follows from the equivalence
\[
\lambda^\tr \notdom \alpha(d,r) \iff \lambda \notlessdom \alpha(d,r)^\tr.
\]
Note that $\alpha(d,r)^\tr = (r+1, 1^{d-r-1}) =
\alpha(d,d-r-1)$. Combining the above with the equality
\begin{equation}
\ann_{\Bbbk W_d} \bV^{\otimes r} = \bigcap_{\lambda \vdash d,
  \, \lambda \dom \alpha(d,r)} \ann_{\Bbbk W_d} M^\lambda
\end{equation}
in Proposition \ref{prop:kernel0-eq}(b), we obtain the following.

\begin{prop}\label{prop:reformulation}
  Assume that $\Bbbk$ is a field of characteristic zero. Let $d =
  n-2\eps$ and assume that $d > r+1$, where $\eps \in \{0,1/2\}$.
  Then
  \[
    \ker \Phi_{n,r+\eps} = \bigcap_{\lambda \vdash d, \, \lambda \dom
      \alpha(d,r)} \ann_{\Bbbk W_d} M^\lambda = \Bbbk \{
    y_{\sts\stt}\mid [\sts]=[\stt] \notlessdom \alpha(d,r)^\tr\}.
  \]
\end{prop}

\begin{proof}
Most of the proof is in the remarks preceding the statement. To
complete the proof, we only need to observe that in the semisimple
case, the cell ideal spanned by all $y_{\sts\stt}$ for $[\sts] = [\stt]
\notlessdom \alpha(d,r)^\tr$ is isomorphic to $\bigoplus_{\lambda
  \notlessdom \alpha(d,r)^\tr} \End_\Bbbk(\Delta(\lambda))$.
\end{proof}

Note the similarity, and the difference, with Lemma 3 of
\cite{Haerterich}. Later we will show that the description of $\ker
\Phi_{n,r+\eps}$ in the rightmost equality of Proposition
\ref{prop:reformulation} is characteristic-free.

The following characterisation of the key inequalities in the above
proposition will be useful in the next section. We write
$\rows(\lambda)$, $\cols(\lambda)$ for the number of rows, columns in
the Young diagram of $\lambda$, respectively.

\begin{lem}\label{lem:kernel-ideal}
  Let $\lambda \vdash d$. Then:

  (a) $\lambda \dom \alpha(d,r) \iff \cols(\lambda) \ge d-r$.

  (b) $\lambda \lessdom \alpha(d,r)^\tr \iff \rows(\lambda) \ge d-r$. 
\end{lem}

\begin{rmk}\label{rmk:b-interp}
  Part (b) implies that $\lambda \notlessdom
  \alpha(d,r)^\tr \iff \rows(\lambda) < d-r$.
\end{rmk}

\begin{proof}
(a) Since $\alpha(d,r) = (d-r, 1^r)$, it follows by the definition of
  the dominance order that $\lambda \dom \alpha(d,r)$ if and only if
\[
  \lambda_1 \ge d-r,\quad \lambda_1+\lambda_2 \ge d-r+1, \quad \dots, \quad
  \lambda_1+ \cdots + \lambda_{r+1} \ge d.
\]
Since $\lambda_1 = \cols(\lambda)$, it follows that $\lambda$ has at least
$d-r$ columns if and only if $\lambda \dom \alpha(d,r)$. This proves (a).  

(b) This follows from (a), since for $\mu,\nu \vdash d$ we have $\mu
\dom \nu \iff \mu^\tr \lessdom \nu^\tr$. First, replace $\lambda$ with $\mu$
in (a) to get $\mu \dom \alpha(d,r) \iff \cols(\mu) \ge d-r$. This
is equivalent to the statement
\begin{align*}
\mu^\tr \lessdom \alpha(d,r)^\tr \iff \cols(\mu) \ge d-r.
\end{align*}
Using the fact that $\cols(\mu) = \rows(\mu^\tr)$, we obtain the desired
result by setting $\lambda = \mu^\tr$ in the above.
\end{proof}

Suppose that $\Omega$ is any set of partitions of $d$. For convenience
of notation, we set $A = \Bbbk W_d$. Following Graham and Lehrer
\cite{Graham-Lehrer}, we define
\[
  A^y[\Omega] = \sum_{\lambda \in \Omega: [\sts] = [\stt] = \lambda}
  \Bbbk y_{\sts\stt} .
\]
A set $\Omega$ of partitions of $d$ is an \emph{ideal} if $\lambda \in
\Omega$ and $\mu \dom \lambda$ (for $\mu \vdash d$) always implies
$\mu \in \Omega$.  It is clear that the set
\begin{equation}\label{fortheref}
  \Omega = \{ \lambda \vdash d \mid \lambda \notlessdom \alpha(d,r)^\tr \}
\end{equation}
is an ideal, because if $\lambda$ is a partition of strictly
fewer than $d-r$ parts and if $\mu \dom \lambda$ then the diagram of
$\mu$ is obtained from the diagram of $\lambda$ by moving boxes up,
which cannot increase the number of rows. Alternatively, if $\lambda
\notlessdom \alpha(d,r)^\tr$ and if $\mu \dom \lambda$ then the
assumption $\mu \lessdom \alpha(d,r)^\tr$ implies that $\lambda \lessdom
\mu \lessdom \alpha(d,r)^\tr$, which forces $\lambda \lessdom
\alpha(d,r)^\tr$, a contradiction.

The importance of this is that Graham and Lehrer proved that the set
$A^y[\Omega]$ is an ideal in $A$ whenever $\Omega$ is an ideal.

\begin{prop}[\cite{Graham-Lehrer}*{Lemma~(1.5)}] \label{prop:GL}
For $\Omega = \{ \lambda \vdash d \mid \lambda \notlessdom
\alpha(d,r)^\tr \}$ as above, 
\[
A^y[\Omega] = A^y[\notlessdom \alpha(d,r)^\tr]
\] 
is a cell ideal in $A = \Bbbk W_d$, with basis $\{y_{\sts \stt}: \sts,
\stt \text{ standard}, [\sts]=[\stt] \notlessdom \alpha(d,r)^\tr\}$.
\end{prop}

The cell ideal $A^y[\notlessdom \alpha(d,r)^\tr]$ is equal to $\ker
\Phi_{n,r+\eps}$ in Proposition \ref{prop:reformulation}.

\section{The annihilator in general}\label{sec:ann-gen}\noindent
Now we revert back to the general case, where $\Bbbk$ is once again an
arbitrary commutative ring. Note that Proposition \ref{prop:GL} holds
in this generality. We continue to set $A = \Bbbk W_d$.

Let $v = \bv_{i_1} \otimes \cdots \otimes \bv_{i_l}$ in $H_d(l)$ be a
simple tensor with distinct tensor factors (i.e., $i_\alpha \ne
i_\beta$ for $1 \le \alpha \ne \beta \le l$). We start by computing
$y_\lambda \cdot v$, where $\lambda \vdash d$ and $1 \le l \le
d$. (See equation \eqref{eq:Murphy-basis} for the definition of
$y_\lambda$.)  Let
\[
W_d^v =\{w \in W_d \mid w\cdot v = v\}
\]
be the stabiliser of $v$ in $W_d$. Clearly we have $W_d^v = W_B$ where
$B = \{1, \dots, d\} \setminus \{i_1, \dots, i_l\}$. Let $W_\lambda$ be the
Young subgroup determined by $\lambda$, and write it as $W_{C_1}
\times \cdots \times W_{C_k}$, where $\lambda$ has $k$ parts. Here
$C_j$ is the subset of $\{1, \dots, d\}$ defined by the numbers in the
$j$th row of $\stt^\lambda$. Note that $\{ C_1, \dots, C_k \}$ is a
set partition of $\{1, \dots, d\}$. We have
\begin{equation}\label{eq:stab-of-v}
  W_\lambda \cap W_B = (W_{C_1} \times \cdots \times W_{C_k}) \cap W_B
  = W_{C_1 \cap B} \times \cdots \times W_{C_k \cap B}.
\end{equation}
The stabiliser of $v$ in $W_\lambda$ is $S = W_\lambda \cap W_d^v =
W_\lambda \cap W_B$ as above. Fix a left coset decomposition
\[
w_1 S \sqcup w_2 S \sqcup \cdots \sqcup w_t S = W_\lambda
\]
of $W_\lambda/S$, where $\{w_1, \dots, w_t\}$ is a set of coset
representatives. Then by the definition of $y_\lambda$ we have
\begin{equation*}
  y_\lambda \cdot v = \sum_{i=1}^t \sum_{s \in S} \sgn(w_is)\, w_is \cdot v
  = \sum_{i=1}^t \sgn(w_i) \big( \sum_{s \in S} \sgn(s) \big)\, w_i \cdot v.
\end{equation*}
This proves that if we set  $F_S = \sum_{s \in S} \sgn(s)$ then we have
\begin{equation}
  y_\lambda \cdot v = F_S \, \sum_{i=1}^t \sgn(w_i) \, w_i \cdot v.
\end{equation}
We note that $F_S$ is a scalar depending only on $S$. Since $S$ is a
product of symmetric groups by \eqref{eq:stab-of-v}, it follows that
$F_S$ is either zero or one, with the latter case occurring if and only
if $S$ is the trivial subgroup. We have shown that:
\begin{equation}\label{eq:when-zero}
  S = W_d^v \cap W_\lambda \ne \{1\} \implies y_\lambda \cdot v = 0. 
\end{equation}

From the above analysis, we now obtain a lower bound for the
annihilator of $\bV^{\otimes r}$ as a $\Bbbk W_d$-module, where
$d=n-2\eps$ as usual. Later we will show the bound is precise.

\begin{prop}\label{prop:containment}
Let $\Bbbk$ be a commutative ring and assume that $d>r+1$. Let
$d=n-2\eps$, where $\eps \in \{0,1/2\}$.  Then the kernel of
$\Phi_{n,r+\eps}$ contains the cell ideal $A^y[\notlessdom
  \alpha(d,r)^\tr]$, where $A = \Bbbk W_d$.
\end{prop}

\begin{proof} 
Let $\lambda \vdash d$ have fewer than $d-r$ parts. In light of
Corollary \ref{cor:decomps}, we need to show that $y_{\sts\stt} \cdot
v = 0$, for any $v \in H_d(l)$, where $1 \le l \le r-2\eps$. Since
$y_{\sts\stt} = d(\sts) y_\lambda d(t)^{-1}$ it clearly suffices to
show that $y_\lambda \cdot v = 0$ for any such $v$.  By
\eqref{eq:when-zero}, this will follow if we can show that $W_d^v \cap
W_\lambda$ is not trivial. By \eqref{eq:stab-of-v}, we have
\[
W_d^v \cap W_\lambda = W_\lambda \cap W_B = W_{C_1\cap B} \times \cdots
\times W_{C_k\cap B}
\]
where $W_\lambda = W_{C_1} \times \cdots \times W_{C_k}$ and $B = \{1,
\dots, d\} \setminus \{i_1, \dots, i_l\}$. Since $\{1, \dots, d\} =
C_1 \sqcup \cdots \sqcup C_k$, we have
\[
B = (C_1 \cap B) \sqcup \cdots \sqcup (C_k \cap B).
\]
Furthermore, we have $|B| = d-l$. The condition $l \le r$ (from above)
forces
\[
|B| = d-l \ge d-r > k.
\]
Hence at least one $C_j \cap B$ has more than one element, and thus
$W_d^v \cap W_\lambda$ is not trivial.
\end{proof}

To obtain the opposite inclusion, we will adapt a result of
H\"{a}rterich.
 
\begin{prop} \label{prop:Haer-for-us}
  Suppose that $\Bbbk$ is a commutative ring.  For any $\mu \vdash d$, 
  \[
    \bigcap_{\lambda \vdash d, \, \lambda \dom \mu} \ann_{A}
    M^{\lambda} = A^y[\notlessdom \mu^\tr], \text{with basis }
    \{y_{\sts\stt} \mid \sts, \stt \text{ standard}, [\sts] =[\stt]
    \notlessdom \mu^\tr \}.
  \]
\end{prop}

\begin{proof}
  (Compare with the proof of \cite{Haerterich}*{Lemma 3}. We are
  setting $q=1$ and replacing right modules with left ones.)
Let $\Omega$ be the set of all $\lambda \vdash d$ such that $\lambda
\notlessdom \mu^\tr$. 
 The set $\Omega $ is that of equation \eqref{fortheref},   
where we remark that this
 is an ideal in the poset of partitions of $d$.  
 \color{black}
 It follows from
\cite{Murphy95}*{Lemma 4.12} that $A^y[\Omega]$ is contained in
$\bigcap_{\lambda \vdash d, \, \lambda \dom \mu} \ann_{A}
M^{\lambda}$, so we need only prove the reverse containment.

Suppose that $h^\# = \sum_{(\sts,\stt)} \alpha_{\sts\stt}
y_{\sts\stt}$ belongs to $\bigcap_{\lambda \vdash d, \, \lambda \dom
  \mu} \ann_{A} M^{\lambda}$. We need to show that $\alpha_{\sts\stt}
\ne 0$ implies $[\sts]=[\stt] \notlessdom \mu^\tr$. Let $(\sts_0,
\stt_0)$ be a \emph{minimal} pair (with respect to $\dom$) such that
$\alpha_{\sts_0\stt_0} \ne 0$. Then $\alpha_{\sts\stt} = 0$ for all
pairs $(\sts,\stt) \slessdom(\sts_0,\stt_0)$, and since the set
$\Omega$ is an ideal in the poset, it suffices to show that $[\sts_0]
= [\stt_0] \notlessdom \mu^\tr$.

Let $\lambda_0^\tr = [\sts] = [\stt]$. The calculation in the proof of
\cite{Haerterich}*{Lemma 3} shows that $h^\# x_{\sts_0^\tr\lambda_0}
\ne 0$.  Since $h^\# \in \bigcap_{\lambda \vdash d, \, \lambda \dom
  \mu} \ann_{A} M^{\lambda}$, it follows that $\lambda_0 \notdom
\mu$. Equivalently, $\lambda_0^\tr \notlessdom \mu^\tr$, so $[\sts_0]
= [\stt_0] \notlessdom \mu^\tr$, as required.
\end{proof}

To finish, we will specialise $\mu$ to $\alpha(d,r)$ in Proposition
\ref{prop:Haer-for-us}. We also need one more fact. We need to show
that the intersection of annihilators in Proposition
\ref{prop:Haer-for-us} in the case $\mu = \alpha(d,r)$ is in fact the
same as the annihilator of $\bV^{\otimes r}$. This is far from
obvious, although we already proved this is so if $\Bbbk$ is a field
of characteristic zero. If Lemma \ref{lem:embedding-ss-case} were true
over any commutative ground ring $\Bbbk$ then we would be able to
apply the same argument that produced Proposition
\ref{prop:kernel0-eq}(a), but unfortunately there is an example in
\cite{Doty-Nyman} showing that Lemma \ref{lem:embedding-ss-case} can
fail in positive characteristic.

It turns out, however, that there is a version of Lemma
\ref{lem:embedding-ss-case}, valid for any commutative ring $\Bbbk$, in
the special case where the second partition is taken to be
$\alpha(d,r)$.

\begin{prop}\label{prop:embedding-hook}
  For any commutative ring $\Bbbk$, and any $\lambda \vdash d$ such
  that $\lambda \dom \alpha(d,r)$, we have an embedding $M^\lambda
  \subseteq M^{\alpha(d,r)}$, as $\Bbbk W_d$-modules.
\end{prop}

\begin{proof}
Write $A = \Bbbk W_d$ as before. For any $\lambda \vdash d$, the
permutation module $M^\lambda$ is isomorphic to the left ideal $A
x_\lambda$. For ease of typography set $\mu = \alpha(d,r)$, and assume
that $\lambda \dom \mu$. Then by applying \cite{Murphy95}*{Lemma 4.1}
we see that $x_\mu$ is a left factor of $x_\lambda = x_{\stt\stt}$ for
$\stt = \stt^\lambda$ the canonical $\lambda$-tableaux; i.e., there is
some $z \in A$ such that $x_\lambda = z x_\mu$. This immediately
implies that $Ax_\lambda = Az x_\mu \subseteq A x_\mu$, as required.
\end{proof}

Now we finally obtain our main result. Recall that $\alpha(d,r)=(d-r,1^r)$.

\begin{thm}\label{thm:main}
Let $d=n-2\eps$, where $\eps \in \{0,1/2\}$.  For any commutative ring
$\Bbbk$, set $A = \Bbbk W_d$.
\begin{enumerate}\renewcommand{\labelenumi}{(\alph{enumi})}
\item The kernel of $\Phi_{n,r+\eps}$ over $\Bbbk$ is equal to the
  cell ideal
\[
A^y[\notlessdom \alpha(d,r)^\tr] = \Bbbk\{y_{\sts\stt} \mid \sts, \stt
\text{ standard}, [\sts] =[\stt] \notlessdom\alpha(d,r)^\tr \},
\]
that is, the $\Bbbk$-span of all $y_{\sts\stt} = y^\lambda_{\sts\stt}$
for which $\rows(\lambda) < d-r$. This spanning set is a cellular
basis of the ideal, which is thus free over $\Bbbk$ of rank which is
independent of $\Bbbk$.

\item The image in $A/(\ker \Phi_{n,r+\eps})$ of the set
\[
\Bbbk\{y_{\sts\stt} \mid \sts, \stt \text{ standard}, [\sts] =[\stt]
\lessdom\alpha(d,r)' \},
\]
that is, the $\Bbbk$-span of all $y_{\sts\stt} = y^\lambda_{\sts\stt}$
for which $\rows(\lambda) \ge d-r$, is a cellular basis of $\im
\Phi_{n,r+\eps}$, which is thus free over $\Bbbk$ of rank which is
independent of $\Bbbk$.

\item $\ker \Phi_{n,r+\eps} = \ann_{\Bbbk W_d} \bV^{\otimes r} =
  \ann_{\Bbbk W_d} M^{\alpha(d,r)}$.
\end{enumerate}
\end{thm}

\begin{proof}
(a) This is just a matter of putting the various pieces together. We
  use Proposition \ref{prop:embedding-hook} to show that the
  annihilator
\[
   \ann_A \bV^{\otimes r} = \bigcap_{\lambda \in \mathcal{H}(d,r)}
   \ann_A M^\lambda 
\] 
remains unchanged when we include extra terms of the form $\ann_A
M^\lambda$ for any $\lambda \dom \alpha(d,r)$. This implies that 
\[
   \ann_A \bV^{\otimes r} = \bigcap_{\lambda \vdash d, \,
     \lambda \dom \alpha(d,r)} \ann_A M^\lambda.
\]
Now the first claim in (a) follows from Proposition
\ref{prop:Haer-for-us}. The equality of the displayed set and its
alternative description as the $\Bbbk$-span of the
$y^\lambda_{\sts\stt}$ for which $\rows(\lambda) < d-r$ is Lemma
\ref{lem:kernel-ideal}; see Remark \ref{rmk:b-interp}.

(b) is an immediate consequence of (a) and the structure of cellular
algebras.

(c) follows from Proposition \ref{prop:embedding-hook}.
\end{proof}

\begin{example}Let $d=n-2\eps$, where $\eps \in \{0,1/2\}$.

(i) If $d-r \le 1$ then $\ker \Phi_{n,r+\eps} = 0$, by Corollary
  \ref{cor:faithful}.

(ii) If $d-r = 2$ then $\ker \Phi_{n,r+\eps} = \Bbbk
  y^{(d)}_{\stt\stt} = \Bbbk\sum_{w \in W_n} \sgn(w)\, w$, of rank 1,
  where $\stt=\stt^{(d)}$ is the unique standard tableau of shape
  $(d)$.

(iii) If $d-r = 3$ then $\ker \Phi_{n,r+\eps}$ is the $\Bbbk$-span of
  all $y^\lambda_{\sts\stt}$ (where $[\sts]=[\stt]=\lambda$) for which
  $\lambda$ has at most two rows.
\end{example}
%
%

By combining Theorem \ref{thm:main} with Schur--Weyl duality, we
obtain the following consequence.

\begin{cor}\label{cor:main}
  Let $\eps \in \{0,1/2\}$. For any commutative ring $\Bbbk$, the
  centraliser algebra $\End_{\Ptn_{r+\eps}(n)} (\bV^{\otimes r})$ is cellular.
\end{cor}

\begin{proof}
By the main result of \cite{BDM:SWD}, that Schur--Weyl duality holds
over a commutative ring $\Bbbk$, the representation
$\Phi_{n,r+\eps}$ surjects onto the algebra
$\End_{\Ptn_{r+\eps}(n)}(\bV^{\otimes r})$. By Theorem
\ref{thm:main}, that image is cellular.
\end{proof}

\begin{rmk}
Stability phenomena for the symmetric group in characteristic zero are
well-documented (see for example \cites{BDO15,Deligne}).  In
characteristic $p>0$ much less is known, but instances of such
phenomena have already been discovered using partition algebra ideas
\cite{MW}.  We hope that the Schur--Weyl duality of
\cite{BDM:SWD} and resulting Schur functors will provide the rigorous
framework necessary to explore these ideas further.
\end{rmk}

\medskip

\noindent{\bf Acknowledgements: } 
The first author is grateful for funding from EPSRC fellowship  grant EP/V00090X/1.

\end{document}